\documentclass{amsart}
\usepackage{amsmath}
\usepackage{amsthm}
\usepackage{amssymb}
\usepackage[backend=biber, style=numeric, sorting=none]{biblatex}
\usepackage{listings}
\usepackage{xcolor}
\lstdefinestyle{sagestyle}{
    language=Python,
    commentstyle=\color{green!40!black},
    keywordstyle=\color{blue},
    numberstyle=\tiny\color{gray},
    stringstyle=\color{purple},
    basicstyle=\ttfamily\footnotesize,
    breaklines=true,
    captionpos=b,
    keepspaces=true,
    numbers=left,
    numbersep=5pt,
    showspaces=false,
    showstringspaces=false,
    showtabs=false,
    tabsize=4
}

\title{A Novel Summation Formula for the Hurwitz-Kronecker Class Number}
\author{Tsai Yi-Ju }
\newtheorem{theorem}{Theorem}
\newtheorem{example}{Example}
\newtheorem{lemma}{Lemma}
\begin{document}

\begin{abstract}
The purpose of this paper is to present a novel and elegant summation formula for $H_w$, the Kronecker-Hurwitz class number. Specifically, for any prime $p$, we have the formula:
$$
\sum_{t^2<p} H_w(t^2-p) = \frac{p-2}{3}.
$$
\end{abstract}
\maketitle

\section{Introduction}
The study of the sum of class numbers is a classical topic. For the form of $H_w(t^2-4p)$, summation formulas are well-understood; a celebrated result, for instance, is that 
$
\sum_{t^2<4p} H_w(t^2-4p) = 2p.
$
When one considers the seemingly similar $H_w(t^2-p)$, however, the literature becomes remarkably sparse. We have located only one publication featuring such a summation, and the summation formula holds only under certain constraints (See Theorem 5 in~\cite{BROWN20081847}). The formula that is the subject of this paper arose serendipitously from the our work in cryptography~\cite{cryptoeprint:2025/1568}.

\section{The Summation Formula}
\begin{theorem}
For any prime $p$, the following formula holds:
\[
\sum_{t^2<p} H_w(t^2-p) = \frac{p-2}{3}.
\]
The detailed definition of $H_w$ is provided in Appendix A. 
\end{theorem}
\begin{example}
   For $p=5,$

   $$\sum_{t^2<p} H_w(t^2-p) = H_w(-1) + H_w(-4) + H_w(-5) + H_w(-4) + H_w(-1) = 1.$$

   Note:  The formula has been verified for the first 10,000 primes using the SageMath code in Appendix B.
\end{example}
\section{The Proof}
\begin{lemma}
Let $p>3$ be a prime. The cardinality of the set $S_M(p,t) = \{ (A,B) \in \mathbb{F}_p^2 \mid |M_{A,B}(\mathbb{F}_p)| = p+1-t \text{ and } B(A^2-4) \neq 0 \}$ is given by the following formula:
\[
|S_M(p,t)| = 
\begin{cases} 
3(p-1)H_w\left(\frac{t^2-4p}{4}\right) & \text{if } 4 \mid p+1-t \text{ and } t^2 < 4p \\
0 & \text{otherwise}
\end{cases},
\]
where $M_{A,B}$ denote a Montgomery curve given by $By^2=x^3+Ax^2+x$ and $H_w$ is the Kronecker-Hurwitz class number.
\end{lemma}
\begin{proof}
See Theorem 6 in~\cite{cryptoeprint:2025/1568}.
\end{proof}
We now present the proof of Theorem 1.
\begin{proof}

For $p\le3$, the formula can be verified directly. For a prime $p>3$, consider any pair $(A,B) \in \left\{ (A,B) \in \mathbb{F}_p^2 \mid B(A^2-4) \neq 0 \right\}$ defining a non-singular curve. Let $t = p+1 - |M_{A,B}(\mathbb{F}_p)|$. By Hasse's theorem, $t^2 < 4p$, and it is a known property of Montgomery curves that $4 \mid (p+1-t)$. This means that any such pair $(A,B)$ belongs to a set $S_M(p,t)$ where $t$ satisfies the conditions of Lemma 1.

Since a given curve has a unique number of points, its trace $t$ is also unique. Therefore, the sets $S_M(p,t)$ are disjoint for different values of $t$. This establishes that the set of all pairs is the disjoint union of the sets $S_M(p,t)$:
\[
\left\{ (A,B) \in \mathbb{F}_p^2 \mid B(A^2-4) \neq 0 \right\} = \bigcup_{\substack{t^2 < 4p \\ 4 \mid p+1-t}} S_M(p,t).
\]
By taking the cardinality of this union, we get:
$$
(p-1)(p-2) = \sum_{t} |S_M(p,t)| 
= \sum_{\substack{t^2 < 4p \\ 4 \mid p+1-t}} 3(p-1)H_w\left(\frac{t^2-4p}{4}\right).
$$
When $p \equiv 1 \pmod 4$,

\begin{align*}
\sum_{\substack{t^2 < 4p \\ 4 \mid p+1-t}} H_w\left(\frac{t^2-4p}{4}\right) 
&= \sum_{\substack{t^2 < 4p \\ t \equiv 2 \pmod 4}} H_w\left(\frac{t^2-4p}{4}\right) \\
&= \sum_{(4t+2)^2 < 4p} H_w\left(\frac{(4t+2)^2-4p}{4}\right) \\
&= \sum_{(2t+1)^2 < p} H_w\left((2t+1)^2-p\right) \\
&= \sum_{\substack{t^2 < p \\ t \equiv 1 \pmod 2}} H_w\left(t^2-p\right).
\end{align*}
Furthermore, for $p \equiv 1 \pmod{4}$ and even $t$, we have $t^2-p \equiv 0-1 \equiv 3 \pmod 4$. Since $t^2-p$ is an odd integer, any of its square divisors $d^2$ is also odd, which implies $d^2 \equiv 1 \pmod 4$. Thus, the quotient satisfies
\[
\frac{t^2-p}{d^2} \equiv \frac{3}{1} \equiv 3 \pmod 4.
\]
By the definition of the Hurwitz class number, $H_w(t^2-p) = \sum_{d^2 | t^2-p} h_w\left(\frac{t^2-p}{d^2}\right)$. The class number of an order $h_w(k)$ is zero for any negative integer $k \equiv 3 \pmod 4$. It follows that $H_w(t^2-p) = 0$.
Therefore, the sum over all even $t$ is zero:
\[
\sum_{\substack{t^2 < p \\ t \equiv 0 \pmod 2}} H_w(t^2 - p) = 0.
\]

When,  $p \equiv 3 \pmod 4$,
\begin{align*}
\sum_{\substack{t^2 < 4p \\ 4 \mid p+1-t}} H_w\left(\frac{t^2-4p}{4}\right) 
&= \sum_{\substack{t^2 < 4p \\ t \equiv 0 \pmod 4}} H_w\left(\frac{t^2-4p}{4}\right) \\
&= \sum_{(4t)^2 < 4p} H_w\left(\frac{(4t)^2-4p}{4}\right) \\
&= \sum_{(2t)^2 < p} H_w\left((2t)^2-p\right) \\
&= \sum_{\substack{t^2 < p \\ t \equiv 0 \pmod 2}} H_w\left(t^2-p\right)
\end{align*}

Additionally, for $p \equiv 3 \pmod{4}$ and odd $t$, we have $t^2-p \equiv 1-3 \equiv 2 \pmod 4$. The integer $t^2-p$ is divisible by 2 but not by 4. Consequently, any of its square divisors $d^2$ is odd, which implies $d^2 \equiv 1 \pmod 4$. Thus, the quotient satisfies:
\[
\frac{t^2-p}{d^2} \equiv \frac{2}{1} \equiv 2 \pmod 4.
\]
Similarly to the previous case, we have $H_w(t^2-p) = 0$.
Therefore, the sum over all odd $t$ is zero:
\[
\sum_{\substack{t^2 < p \\ t \equiv 1 \pmod 2}} H_w(t^2 - p) = 0.
\]
Thus, 
\begin{align*}
\sum_{t^2<p} H_w(t^2-p) &= \sum_{\substack{t^2<p  \\  t \equiv 0 \pmod 2}} H_w(t^2-p) +  \sum_{\substack{t^2<p  \\  t \equiv 1 \pmod 2}} H_w(t^2-p) \\ 
&=0+\sum_{\substack{t^2 < 4p \\ 4 \mid p+1-t}} H_w\left(\frac{t^2-4p}{4}\right)=\frac{p-2}{3}.
\end{align*}
\end{proof}
\appendix
\newpage
\section{}
The definitions presented here are excerpted directly from~\cite{https://doi.org/10.1112/plms.12069}, as shown below. It is important to note that the definition of $h_w(d)$ for $d \equiv 2, 3 \pmod{4}$ varies in the literature. While some sources leave it undefined, others set it to zero. We adopt the latter convention, as in~\cite{https://doi.org/10.1112/plms.12069} and ~\cite{ZAGIER1974-1975}.

For $d < 0$ with $d \equiv 0, 1 \pmod{4}$, let $h(d)$ denote the class number of the unique quadratic order of discriminant $d$. Let
$$
h_w(d) \overset{\text{def}}{=}
\begin{cases}
    h(d)/3, & \text{if } d = -3, \\
    h(d)/2, & \text{if } d = -4, \\
    h(d),   & \text{if } d < 0, d \equiv 0,1 \pmod{4}, \text{ and } d \neq -3, -4, \\
    0,      & \text{otherwise}
\end{cases}
$$
and for $\Delta<0$ let

$$
H(\Delta) \overset{\text{def}}{=} \sum_{d^2 \mid \Delta} h_w\left(\frac{\Delta}{d^2}\right).
$$
\newpage
\section{}
The following SageMath code was used to verify the summation formula for the first 10,000 primes.
\begin{lstlisting}[style=sagestyle, caption={SageMath code to verify Theorem 1.}]
def h_w(d):
    if d == -3:
        return 1/3
    elif d == -4:
        return 1/2
    elif ((d<0) and (((((d%4)+4)%4) == 0) or ((((d%4)+4)%4) == 1))):
        return BQFClassGroup(d).order()
    else:
        return 0
    return 0
def H_w(D):
    if D == 0:
        return -1/12
    if D > 0:
        return 0
    s = 0
    div_list = divisors(D)
    square_factor_list = [_ for _ in div_list if is_square(_)]
    for i in square_factor_list:
        s += h_w(D//i)
    return s
def sum_of_H_w(p):
    s = 0
    for i in range(-ceil(sqrt(p)),ceil(sqrt(p))):
        s += H_w(i^2-p)
    return s
bool_list=[]
for p in prime_range(1, 104730):
    bool_list.append(3*sum_of_H_w(p) == (p-2))
from collections import Counter
print(Counter(bool_list))
\end{lstlisting}
\newpage
\section*{Acknowledgment}
We would like to thank Gemini for its assistance in translating the manuscript. The core content of this paper, including the proofs, code, and main theorems, was originally written in our native language and subsequently reviewed for accuracy after translation.
\printbibliography

@misc{cryptoeprint:2025/1568,
      author = {Tsai Yi-Ju},
      title = {Montgomery Curves: Exact Enumeration and Probabilistic Analysis},
      howpublished = {Cryptology {ePrint} Archive, Paper 2025/1568},
      year = {2025},
      url = {https://eprint.iacr.org/2025/1568}
}

@article{BROWN20081847,
title = {Elliptic curves, modular forms, and sums of Hurwitz class numbers},
journal = {Journal of Number Theory},
volume = {128},
number = {6},
pages = {1847-1863},
year = {2008},
issn = {0022-314X},
doi = {https://doi.org/10.1016/j.jnt.2007.10.008},
author = {Brittany Brown and Neil J. Calkin and Timothy B. Flowers and Kevin James and Ethan Smith and Amy Stout}
}

@article{ZAGIER1974-1975,
author = {ZAGIER, Don},
journal = {Seminaire de Théorie des Nombres de Bordeaux},
pages = {1-2},
title = {Nombres de classes et formes modulaires de poids 3/2.},
url = {http://eudml.org/doc/181964},
volume = {4},
year = {1974-1975},
}

@article{https://doi.org/10.1112/plms.12069,
author = {Kaplan, Nathan and Petrow, Ian},
title = {Elliptic curves over a finite field and the trace formula},
journal = {Proceedings of the London Mathematical Society},
volume = {115},
number = {6},
pages = {1317-1372},
doi = {https://doi.org/10.1112/plms.12069},
year = {2017}
}
\end{document}